\documentclass[11pt,a4paper,reqno]{amsart} 
\usepackage[applemac]{inputenc}
\usepackage{amsmath}
\usepackage{amssymb}
\usepackage{euscript}
\usepackage[colorlinks=true,linktocpage=true,pagebackref=false, citecolor=black,linkcolor=black]{hyperref}
\usepackage{pdfsync}
\usepackage{mathrsfs}
\usepackage{wasysym}
\usepackage[all]{xy}
\usepackage{color}

\newtheorem{thm}{Theorem}[section]
\newtheorem{lem}[thm]{Lemma}
\newtheorem{prop}[thm]{Proposition}

\newtheorem{cor}[thm]{Corollary}

\theoremstyle{definition}

\newtheorem{examples}[thm]{Examples}

\theoremstyle{remark}
\newtheorem{remark}[thm]{Remark}
\newtheorem{remarks}[thm]{Remarks}

\numberwithin{equation}{section}

 \newcommand{\R}{{\mathbb R}}
\newcommand{\Q}{{\mathbb Q}}

\newcommand{\sph}{{\mathbb S}}

\newcommand{\gtp}{{\mathfrak p}} \newcommand{\gtq}{{\mathfrak q}}
\newcommand{\gtm}{{\mathfrak m}} \newcommand{\gtn}{{\mathfrak n}}
\newcommand{\gta}{{\mathfrak a}}

\newcommand{\Bb}{{\EuScript B}}

\newcommand{\Dd}{{\EuScript D}}
\newcommand{\Ee}{{\EuScript E}}

\newcommand{\Zz}{{\EuScript Z}}
\newcommand{\Tt}{{\EuScript T}}

\newcommand{\Hom}{\operatorname{Hom}}
\newcommand{\im}{\operatorname{im}}
\newcommand{\qf}{\operatorname{qf}}

\newcommand{\Int}{\operatorname{Int}}
\newcommand{\dist}{\operatorname{dist}}

\newcommand{\hgt}{\operatorname{ht}}

\newcommand{\ceros}{\operatorname{\mathcal Z}}

\newcommand{\gr}{\operatorname{graph}}
\newcommand{\cl}{\operatorname{Cl}}

\newcommand{\di}{\operatorname{{\mathcal D}}}

\newcommand{\dgt}{\operatorname{d}}

\newcommand{\lc}{\operatorname{lc}}
\newcommand{\diam}{{\text{\tiny$\displaystyle\diamond$}}}
\newcommand{\gtmd}{\operatorname{\gtm^{\diam}\hspace{-1.5mm}}}

\newcommand{\Specs}{\operatorname{Spec_s}}
\newcommand{\Speca}{\operatorname{Spec_s^*}}
\newcommand{\Specd}{\operatorname{Spec_s^{\diam}}}
\newcommand{\betas}{\operatorname{\beta_s\!}}
\newcommand{\betaa}{\operatorname{\beta_s^*\!\!}}
\newcommand{\betad}{\operatorname{\beta_s^{\diam}\!\!}}

\newcommand{\x}{{\tt x}} \newcommand{\y}{{\tt y}} 
\newcommand{\z}{{\tt z}} \renewcommand{\t}{{\tt t}}
\newcommand{\s}{{\tt s}}

\newcommand{\veps}{\varepsilon}

\newcommand{\ol }{\overline}

\begin{document}

\title[On the remainder of the semialgebraic Stone-C\v{e}ch compactification]{On the remainder of the semialgebraic\\ Stone-C\v{e}ch compactification of a semialgebraic set}

\author{Jos\'e F. Fernando}
\author{J.M. Gamboa}
\address{Departamento de \'Algebra, Facultad de Ciencias Matem\'aticas, Universidad Complutense de Madrid, 28040 MADRID (SPAIN)}
\curraddr{}
\email{josefer@mat.ucm.es, jmgamboa@mat.ucm.es}
\thanks{Authors supported by Spanish GR MTM2011-22435 and GAAR Grupos UCM 910444}

\date{March 31, 2015}


\begin{abstract}
In this work we analyze some topological properties of the remainder $\partial M:=\betaa M\setminus M$ of the semialgebraic Stone-C\v{e}ch compactification $\betaa M$ of a semialgebraic set $M\subset\R^m$ in order to `distinguish' its points from those of $M$. To that end we prove that the set of points of $\betaa M$ that admit a metrizable neighborhood in $\betaa M$ equals $M_{\lc}\cup(\cl_{\betaa M}(\ol{M}_{\leq1})\setminus\ol{M}_{\leq1})$ where $M_{\lc}$ is the largest locally compact dense subset of $M$ and $\ol{M}_{\leq1}$ is the closure in $M$ of the set of $1$-dimensional points of $M$. In addition, we analyze the properties of the sets $\widehat{\partial}M$ and $\widetilde{\partial}M$ of free maximal ideals associated with formal and semialgebraic paths. We prove that both are dense subsets of the remainder $\partial M$ and that the differences $\partial M\setminus\widehat{\partial}M$ and $\widehat{\partial} M\setminus\widetilde{\partial}M$ are also dense subsets of $\partial M$. It holds moreover that all the points of $\widehat{\partial}M$ have countable systems of neighborhoods in $\betaa M$.
\end{abstract}

\maketitle

\section{Introduction}\label{s1}

A semialgebraic set $M\subset\R^m$ is a (finite) boolean combination of sets defined by polynomial equalities and inequalities. A continuous map $f:M\to N$ is \emph{semialgebraic} if its graph is a semialgebraic set (in particular, both $M$ and $N$ are semialgebraic). As usual $f$ is a \em semialgebraic function \em if $N=\R$. Denote the set of (continuous) semialgebraic maps from $M$ to $N$ with ${\mathcal S}(M,N)$. The sum and product of functions defined pointwise endow the set ${\mathcal S}(M)$ of semialgebraic functions on $M$ with a natural structure of a unital commutative ring. In fact ${\mathcal S}(M)$ is an $\R$-algebra and the subset ${\mathcal S}^*(M)$ of bounded semialgebraic functions on $M$ is an $\R$-subalgebra of ${\mathcal S}(M)$. Write ${\mathcal S}^{\diam}(M)$ to refer indistinctly either to ${\mathcal S}(M)$ or ${\mathcal S}^*(M)$. We denote the Zariski spectrum of ${\mathcal S}^{\diam}(M)$ with $\Specd(M)$ and the maximal spectrum of ${\mathcal S}^{\diam}(M)$ with $\betad M$. The maximal spectra $\betas M$ and $\betaa M$ are always homeomorphic but the involved homeomorphism is not natural from a categorical point of view \cite[3.6]{fg3}. In the following $M\subset\R^m$ and $N\subset\R^n$ always denote semialgebraic sets.

A point $p$ of a topological space $X$ is an \em endpoint of $X$ \em if it has an open neighborhood $U\subset X$ equipped with a homeomorphism $f:U\to[0,1)$ that maps $p$ onto $0$. In case $X$ is a semialgebraic set we may assume that the previous homeomorphism is semialgebraic. We denote the set of endpoints of $X$ with $\eta(X)$. Recall briefly how the ring ${\mathcal S}(M)$ (resp. ${\mathcal S}^*(M)$) determine $M$ (resp. $M$ besides the finite set $\eta(M)$ of endpoints of $M$) up to semialgebraic homeomorphism. More generally the Zariski spectrum $\Specs(M)$ determines $M$ up to homeomorphism while $\Speca(M)$ classifies topologically $M\setminus\eta(M)$.

\subsection{Rings of semialgebraic functions}\label{11}
It is natural to wonder whether the ring ${\mathcal S}(M)$ determines the semialgebraic set $M$. Let $N\subset\R^n$ be a semialgebraic set, let $\varphi:{\mathcal S}(N)\to{\mathcal S}(M)$ be a homomorphism of $\R$-algebras and let $\pi_i$ be the restriction to $N$ of the $i$th projection. Recall that: \em the image of the semialgebraic map $f:=(f_1,\ldots,f_n):=(\varphi(\pi_1),\ldots,\varphi(\pi_n)):M\to\R^n$ is contained in $N$ and $\varphi(g)=g\circ f$ for all $g\in{\mathcal S}(N)$\em. 
\begin{proof}
Fix $a\in M$ and let $g_a\in{\mathcal S}(N)$ be given by $g_a(x_1,\ldots,x_n):=\sum_{i=1}^n(x_i-f_i(a))^2$. Observe that $\varphi(g_a)$ vanishes at $a$, so $\varphi(g_a)$ is not a unit of ${\mathcal S}(M)$. Consequently, $g_a$ has a (unique) zero in $N$, which is $f(a)$. Thus, $f(a)\in N$. Denote $\gtm_a$ the maximal ideal of ${\mathcal S}(M)$ associated with $a$ and $\gtn_{f(a)}$ the maximal ideal of ${\mathcal S}(N)$ associated with $f(a)$. As $\varphi(g_a)\in\gtm_a$, also $g_a^2+h^2\in\varphi^{-1 }(\gtm_a)$ for every $h\in\varphi^{-1 }(\gtm_a)$. As $g_a^2+h^2$ is not a unit, it vanishes at the unique zero $f(a)$ of $g_a$. Consequently, $\varphi^{-1}(\gtm_a)\subset\gtn_{f(a)}$, so $\varphi(h)(a)=0$ implies $h(f(a))=0$ for all $h\in{\mathcal S}(N)$. As $\varphi$ is an $\R$-algebra homeomorphism, $\varphi(g)=g\circ f$ for all $g\in{\mathcal S}(N)$, as claimed.
\end{proof}

Consider the natural map 
$$
(\cdot)^*:{\mathcal S}(M,N)\to\Hom_{\R\text{-alg}}({\mathcal S}(N),{\mathcal S}(M)),\ f\mapsto f^*
$$ 
where $f^*:{\mathcal S}(N)\to{\mathcal S}(M),\ g\mapsto g\circ f$. We have proved before: \em $(\cdot)^*$ is a bijection\em. Consequently, \em $M$ and $N$ are semialgebraically homeomorphic if and only if the rings ${\mathcal S}(M)$ and ${\mathcal S}(N)$ are isomorphic\em. This argument goes back to the pioneer work of Schwartz \cite{s0,s1,s2}. Consequently, the category of semialgebraic sets is faithfully reflected in the full subcategory of real closed rings consisting of all $\R$-algebras of the form ${\mathcal S}(M)$. 

\subsection{Rings of bounded semialgebraic functions} 
The next step is to wonder whether the ring ${\mathcal S}^*(M)$ determines the semialgebraic set $M$. A point $p\in M$ is an \em endpoint of $M$ \em if it has an open neighborhood $U\subset M$ equipped with a semialgebraic homeomorphism $f:U\to[0,1)$ that maps $p$ onto $0$. Recall that ${\mathcal S}^{\diam}(M)$ is a real closed ring \cite{s1,sm}. In \cite[\S11]{t1} it is shown that for every real closed ring $A$ there exists a largest real closed ring $B$ such that $A$ is convex in $B$. In \cite{s3} it is shown how the spectrum of a real closed ring lies in the spectrum of any convex subring. Schwartz proved in \cite[\S5]{s4} that ${\mathcal S}(M\setminus\eta(M))$ is the convex closure of ${\mathcal S}^*(M)={\mathcal S}^*(M\setminus\eta(M))$. If 
${\mathcal S}^*(N)$ and ${\mathcal S}^*(M)$ are isomorphic as $\R$-algebras, then their convex closures ${\mathcal S}(N\setminus\eta(N))$ and ${\mathcal S}(M\setminus\eta(M))$ are also isomorphic as $\R$-algebras. Consequently, the semialgebraic sets $M\setminus\eta(M)$ and $N\setminus\eta(N)$ are by \ref{11} semialgebraically homeomorphic. 

\subsection{Homeomorphisms between Zariski spectra}
Homeomorphisms between Zariski spectra induced by $\R$-algebra isomorphisms are quite restrictive and it is natural to wonder what happens when dealing with general homeomorphisms. Let us recall first how the Zariski spectrum $\Specs(M)$ determines $M$ up to homeomorphism. To that end we need the following topological property that distinguishes the points of $M$ from those of $\Specs(M)\setminus M$: 

\paragraph{}\em The maximal ideals $\gtm_a$ of ${\mathcal S}(M)$ associated with points $a\in M$ can be characterized topologically as those points that are isolated for the inverse topology of $\Specs(M)$ within the set of closed points of the Zariski topology of $\Specs(M)$ \em (the inverse topology has the open and quasi-compact sets of $\Specs(M)$ as a basis of closed sets). In algebraic terms this means that the maximal ideals $\gtm_a$ associated with points $a\in M$ are exactly those maximal ideals that are the Jacobson radical of a principal ideal (the ideal generated by the distance function to a point $a\in M$ restricted to $M$ has Jacobson radical $\gtm_a$). Consequently, 

\paragraph{} \em Every homeomorphism $\gamma:\Specs(N)\to\Specs(M)$ restricts to homeomorphisms $\gamma|_{\betas N}:\betas N\to\betas M$ and $\gamma|_N:N\to M$\em. 

There are other topological properties of $M$ that can be explicitly encoded in $\Specs(M)$, see \cite[\S4]{t1}. For an approach to these questions in the frame of rings of definable continuous functions on definable sets in o-minimal expansions of fields we refer the reader to \cite{t0}. 

\paragraph{}\label{restricti0} With the Zariski spectrum $\Speca(M)$ one can proceed analogously using the inverse topology, so we have to determine the set $\Ee$ of maximal ideals (or closed points for the Zariski topology) $\gtm^*$ that are the Jacobson radical of a principal ideal of ${\mathcal S}^*(M)$. We prove in Corollary \ref{eum2} that $\Ee=M\cup\eta(\betaa M)$. As $\eta(M)\subset\eta(\betaa M)$, the Zariski spectrum $\Speca(M)$ classifies topologically $\betaa M$ and $M\setminus\eta(M)$. In this case ${\mathcal S}^*(M)={\mathcal S}^*(M\setminus\eta(M))$, so we cannot distinguish topologically the endpoints of $M$ from those of $\betaa M$ that are not in $M$. Consequently, 

\paragraph{} \em Every homeomorphism $\gamma:\Speca(N)\to\Speca(M)$ restricts to homeomorphisms $\gamma|_{\betas N}:\betas N\to\betas M$ and $\gamma|_{N\setminus\eta(N)}:N\setminus\eta(N)\to M\setminus\eta(M)$\em. 

We prove in Corollary \ref{eum} that $M\setminus\eta(M)$ is the set of closed branching points of $\Speca(M)$. A prime ideal $\gtp$ of ${\mathcal S}^*(M)$ is a \em branching point of $\Speca(M)$ \em if there exist two prime ideals $\gtq_1,\gtq_2$ of ${\mathcal S}^*(M)$ different from $\gtp$ such that $\gtp\in\cl(\gtq_i)$ for $i=1,2$ but $\gtq_i\not\in\cl(\gtq_j)$ if $i\neq j$. The condition `to be a closed branching point' provides an alternative topological characterization of the points of $M\setminus\eta(M)$ inside $\Speca(M)$ that only involves its Zariski topology. 

\paragraph{} As a consequence of \ref{restricti0} it holds that a homeomorphism $\gamma:\betaa N\to\betaa M$ restricts to a homeomorphism $\gamma|_{N\setminus\eta(N)}:N\setminus\eta(N)\to M\setminus\eta(M)$. On the contrary there are many homeomorphisms between semialgebraic sets that do not extend to their semialgebraic Stone-C\v{e}ch compactifications (see Examples \ref{gpn}). Consequently, they admit neither extensions to the Zariski spectra $\Specs(M)$ and $\Speca(M)$. In addition Shiota--Yokoi proposed in \cite{ys} a pair of compact homeomorphic semialgebraic sets that are not semialgebraically homeomorphic.

\subsection{Topological properties of maximal spectra}
It seems reasonable to find alternative topological conditions with respect to the Zariski topology of $\betaa M$ that characterize the points of $M$ among those of $\betaa M$. In \cite[9.6-7]{gj} the authors prove that if $X$ is a metrizable space, then $X$ is the set of $G_{\delta}$-points of the Stone--\v{C}ech compactification $\beta X$ of $X$. It would seem reasonable to follow a similar strategy. As we show in Lemma \ref{cn}, all points of $M$ have a countable basis of neighborhoods in $\betaa M$. We prove in Theorem \ref{gpn3} that the same happens for the dense subset $\widehat{\partial}M$ of the remainder $\partial M:=\betaa M\setminus M$ constituted by the free maximal ideals associated with formal paths that we study with care in Section \ref{s4}. We study also some properties of the set $\widetilde{\partial}M$ constituted by the free maximal ideals associated with semialgebraic paths. We prove that this set is dense in $\partial M$ and that $\widehat{\partial}M\setminus\widetilde{\partial}M$ and $\partial M\setminus\widehat{\partial}M$ are respectively dense in $\widehat{\partial}M$ and $\partial M$.

An almost satisfactory topological property to distinguish the points of $\partial M$ from those of $M$ is `to admit a metrizable neighborhood in $\betaa M$'. We characterized the semialgebraic sets $M$ whose maximal spectrum $\betaa M$ is a metrizable space in \cite[5.17]{fg5}: \em this happens for those semialgebraic sets whose maximal spectrum $\betaa M$ is homeomorphic to a semialgebraic set\em. In Theorem \ref{mn} we prove that the set of points of $\betaa M$ that admit a metrizable neighborhood in $\betaa M$ is $M_{\lc}\cup\eta(\betaa M)$. In addition, $\eta(\betaa M)=\eta(M)\cup(\cl_{\betaa M}(\ol{M}_{\leq1})\setminus\ol{M}_{\leq1})$ where $\ol{M}_{\leq1}$ is the closure in $M$ of the set of $1$-dimensional points of $M$.

For a complementary study of other topological properties of $\partial M$ (as local connectedness, local compactness or number of connected components) we refer the reader to \cite{fg5}. 

\subsection*{Structure of the article} 
In Section \ref{s2} we compile the preliminary terminology and results concerning Zariski and maximal spectra of rings of semialgebraic and bounded semialgebraic functions that we use along this work. Most of the results in Section \ref{s2} are collected from \cite{fe1,fg2,fg3,fg5} and presented without proofs. The reading can be started directly in Section \ref{s3} and referred to the preliminaries only when needed. In Section \ref{s3} we study algebraic properties of points of the remainder associated with formal paths and semialgebraic paths and we analyze as a consequence some announced properties of $\eta(M)$ and $M\setminus\eta(M)$. In Section \ref{s4} we analyze the main properties of the remainder $\partial M$ quoted above: \em density of $\widehat{\partial}M$ in $\partial M$, density of $\widehat{\partial}M\setminus\widetilde{\partial}M$ and $\partial M\setminus\widehat{\partial}M$ in $\widehat{\partial}M$ and $\partial M$ respectively, the points of $\widehat{\partial}M$ are first-countable in $\betaa M$ \em and \em the characterization of the points of $\betaa M$ with metrizable neighborhoods\em.

\section{Preliminaries on spectra of rings of semialgebraic functions}\label{s2}

In the following $M\subset\R^m$ denotes a semialgebraic set. For each $f\in{\mathcal S}^{\diam}(M)$ the semialgebraic set $Z(f):=\{x\in M:\ f(x)=0\}$ is the \em zero set \em of $f$ and $D(f):=M\setminus Z(f)$ is its complement. The \em dimension \em $\dim(M)$ of $M$ is the dimension of its algebraic Zariski closure \cite[\S2.8]{bcr}. The \em local dimension $\dim(M_x)$ of $M$ at a point $x\in\cl(M)$ \em is the dimension $\dim(U)$ of a small enough open semialgebraic neighborhood $U\subset\cl(M)$ of $x$. The dimension of $M$ coincides with the maximum of those local dimensions. For any fixed $d$ the set of points $x\in M$ such that $\dim(M_x)=d$ is a semialgebraic subset of $M$. 

\subsection{Locally closed semialgebraic sets}
Locally closed subsets of a locally compact topological space coincide with locally compact ones \cite[\S9.7. Prop.12-13]{bo}. The sets $\cl(M)$ and $U:=\R^m\setminus(\cl(M)\setminus M)$ are semialgebraic. If $M$ is locally compact, $U$ is open and $M=\cl(M)\cap U$ is the intersection of a closed and an open semialgebraic subsets of $\R^m$. The construction of the largest locally compact and dense subset $M_{\lc}$ of $M$ is the main goal of \cite[9.14-9.21]{dk2}. Define $\rho_0(M):=\cl(M)\setminus M$ and $\rho_1(M):=\rho_0(\rho_0(M))=\cl(\rho_0(M))\cap M$. 

\begin{prop}\label{rho} 
The semialgebraic set $M_{\lc}:=M\setminus\rho_1(M)=\cl(M)\setminus\cl(\rho_0(M))$ is the largest locally compact and dense subset of $M$ and coincides with the set of points of $M$ that have a compact neighborhood in $M$.
\end{prop}

\begin{remarks}\label{curvas}
(i) If $M$ has dimension $\leq 1$, then $M$ is locally compact.

(ii) Denote the set of points of $M$ of local dimension $\geq 2$ with $M_{\geq2}$. We have $M=M_{\geq2}\cup\ol{M}_{\leq1}$ and $M_{\geq2}\cap\ol{M}_{\leq1}$ is a finite set.

(iii) \em If $M_{\geq2}$ is compact, $M$ is locally compact\em. Indeed, 
$$
\rho_0(M)=\cl(M_{\geq2}\cup\ol{M}_{\leq1})\setminus(M_{\geq2}\cup\ol{M}_{\leq1})=\cl(\ol{M}_{\leq1})\setminus(M_{\geq2}\cup\ol{M}_{\leq1})=\rho_0(\ol{M}_{\leq1})\setminus M_{\geq2}
$$
is a finite set, so $\rho_1(M)$ is empty. Thus, $M=M_{\lc}$ is locally compact.
\end{remarks}

\subsection{Zariski spectra of rings of semialgebraic functions.}\label{zr}\setcounter{paragraph}{0}
We recall some results concerning the Zariski spectra of rings of semialgebraic functions and bounded semialgebraic functions \cite[\S3-\S6]{fg3}. The \em Zariski spectrum \em $\Specd(M)$ is endowed with the Zariski topology, which has the family of sets $\Dd_{\Specd(M)}(f):=\{\gtp\in\Specd(M):\, f\not\in\gtp\}$ as a basis of open sets where $f\in{\mathcal S}^{\diam}(M)$. We denote $\Zz_{\Specd(M)}(f):=\Specd(M)\setminus\Dd_{\Specd(M)}(f)$. Recall that ${\mathcal S}^{\diam}(M)$ is a Gelfand ring, so each prime ideal of ${\mathcal S}^{\diam}(M)$ is contained in a unique maximal ideal.

\paragraph{} If $a\in M$, we denote the maximal ideal of all functions in ${\mathcal S}^{\diam}(M)$ vanishing at $a$ with $\gtmd_a$. The map $\phi:M\to\Specd(M),\ a\mapsto\gtmd_a$ embeds $M$ endowed with the Euclidean topology into $\Specd(M)$ as a dense subspace. 

\paragraph{}\label{closedspec}
Given a semialgebraic map $\varphi:N\to M$, there exists a unique continuous map $\Specd(\varphi):\Specd(N)\to\Specd(M)$, which extends $\varphi$. In addition, if $N\subset M$ and $N$ is closed in $M$, then $\Specd(N)\cong\cl_{\Specd(M)}(N)$ via $\Specd({\tt j})$ where ${\tt j}:N\hookrightarrow M$ is the inclusion map.

\paragraph{}{\em Semialgebraic depth.}\label{ht}
Let $\gtp\subset\gtq$ be two prime ideals of ${\mathcal S}^{\diam}(M)$. The \em coheight of $\gtp$ in $\gtq$ \em is the maximum of the integers $r\geq0$ such that there exists a chain of prime ideals $\gtp:=\gtp_0\subsetneq\cdots\subsetneq\gtp_r=:\gtq$. We define the \em coheight of a prime ideal $\gtp$ of ${\mathcal S}^{\diam}(M)$ \em as the coheight of $\gtp$ in the unique maximal ideal of ${\mathcal S}^{\diam}(M)$ containing $\gtp$. 

An ideal $\gta$ of ${\mathcal S}(M)$ is a \em $z$-ideal \em if given $f,g\in {\mathcal S}(M)$ such that $Z(f)\subset Z(g)$ and $f\in\gta$, then $g\in\gta$. If $M$ is locally compact, all prime ideals of ${\mathcal S}(M)$ are $z$-ideals by \cite[2.6.6]{bcr}.

The \em semialgebraic depth \em of a prime ideal $\gtp$ of ${\mathcal S}(M)$ is $\dgt_M(\gtp):=\min\{\dim Z(f):\, f\in\gtp\}$. If $\gtp\subset\gtq$ are prime $z$-ideals of ${\mathcal S}(M)$, the coheight of $\gtp$ in $\gtq$ is $\leq\dgt_M(\gtp)-\dgt_M(\gtq)$ (see \cite[4.14(i)]{fe1}). 

\subsection{Maximal spectra of rings of semialgebraic functions}\label{spectracomp}\setcounter{paragraph}{0}
Denote the collection of all maximal ideals of ${\mathcal S}^{\diam}(M)$ with $\betad M$ and consider in $\betad M$ the topology induced by the Zariski topology of $\Specd(M)$. Given $f\in{\mathcal S}^{\diam}(M)$, we denote ${\mathcal D}_{\betad M}(f):=\Dd_{\Specd(M)}(f)\cap\betad M$ and ${\mathcal Z}_{\betad M}(f):=\betad M\setminus{\mathcal D}_{\betad M}(f)=\Zz_{\Specd(M)}(f)\cap\betad M$. By \cite[7.1.25(ii)]{bcr} $\betad M$ is a Hausdorff compactification of $M$. 

\paragraph{}\label{homeo}
The map $\Phi:\betas M\to\betaa M$ that associates with each maximal ideal $\gtm$ of ${\mathcal S}(M)$ the unique maximal ideal $\gtm^*$ of ${\mathcal S}^*(M)$ that contains the prime ideal $\gtm\cap{\mathcal S}^*(M)$ is a homeomorphism. In particular, $\Phi(\gtm_a)=\gtm_a^*$ for all $a\in M$. We denote the maximal ideals of ${\mathcal S}^*(M)$ with $\gtm^*$ where $\gtm$ is the unique maximal ideal of ${\mathcal S}(M)$ such that $\gtm\cap{\mathcal S}^*(M)\subset\gtm^*$.

\paragraph{}\label{cocr} 
The inclusion map $\R\hookrightarrow{\mathcal S}^*(M)/\gtm^*,\, r\mapsto r+\gtm^*$ is an isomorphism of ordered fields because ${\mathcal S}^*(M)/\gtm^*$ is an Archimedean extension of $\R$. As $\R$ admits a unique automorphism, there is no ambiguity to refer to $f+\gtm^*$ as a real number for every $f\in{\mathcal S}^*(M)$. In particular, we identify $f+\gtm_a^*$ with $f(a)$ for all $a\in M$. Thus, each $f\in{\mathcal S}^*(M)$ defines a (unique) natural extension $\widehat{f}:\betaa M\to\R,\, \gtm^*\to f+\gtm^*$, which is continuous because given real numbers $r<s$, we have $\widehat{f}^{-1}((r, s))={\mathcal D}_{\betaa M}(f-r+|f-r|,s-f+|s-f|)$.

\paragraph{}\label{closedbeta1}
If $\varphi:N\to M$ is a semialgebraic map between semialgebraic sets $N$ and $M$, then $\Speca(\varphi):\Speca(N)\to\Speca(M)$ maps $\betaa N$ into $\betaa M$, so we denote the restriction of $\Speca(\varphi)$ to $\betaa N$ with $\betaa \varphi:\betaa N\to\betaa M$. Let $C,C_1,C_2$ be closed semialgebraic subsets of the semialgebraic set $M$ and ${\tt j}:C\hookrightarrow M$ the inclusion map. Then 
\begin{itemize}
\item[(i)] The space $\betaa\, C$ is homeomorphic to $\cl_{\betaa M}(C)\subset\betaa M$ via $\betaa\,{\tt j}:\betaa C\to\betaa M$. 
\item[(ii)] $\cl_{\betaa M}(C_1\cap C_2)=\cl_{\betaa M}(C_1)\cap\cl_{\betaa M}(C_2)$.
\end{itemize}

\paragraph{}\label{freefixed}
The zero set of a prime ideal $\gtp$ of ${\mathcal S}^{\diam}(M)$ provides no substantial information about $\gtp$ because it is either a singleton or the empty set. An ideal $\gta$ of ${\mathcal S}^{\diam}(M)$ is said to be \em fixed \em if all functions in $\gta$ vanish simultaneously at some point of $M$. Otherwise the ideal $\gta$ is \em free\em. The fixed maximal ideals of the ring ${\mathcal S}^{\diam}(M)$ are those of the form $\gtm_a^{\diam}$ where $a\in M$. The equality $\gtm\cap{\mathcal S}^*(M)=\gtm^*$ characterizes the fixed maxi\-mal ideals of ${\mathcal S}^{\diam}(M)$ (see \cite[3.7]{fg5}). Namely,
$$
\text{$\gtm^*$ is a fixed ideal $\iff$ $\gtm$ is a fixed ideal $\iff$ $\gtm\cap{\mathcal S}^*(M)=\gtm^*$ $\iff$ $\hgt(\gtm)=\hgt(\gtm^*)$}.
$$

\section{Points of the remainder associated with formal and semialgebraic paths}\label{s3}

In this section we analyze some topological properties of the set of the points of the remainder $\partial M:=\beta_s^*M\setminus M$ associated with formal paths. For simplicity we assume $M\subset\R^m$ bounded.

\subsection{Extension of coefficients}\label{eoc} 
Let $F$ be a real closed field containing $\R$. There exists a (unique) semialgebraic subset $M_{F}\subset F^m$ called \em extension of $M$ to $F$ \em that satisfies $M=M_{F}\cap\R^m$. The extension of semialgebraic sets depicts the expected behavior with respect to boolean and topological operations, Transfer Principle, etc. \cite[\S5.1-3]{bcr}. Given a semialgebraic map $f:M\to N$, there exists a unique semialgebraic map $f_{F}:M_{F}\to N_{F}$ called \em extension of $f$ to $F$ \em that fulfills $f_{F}|_M=f$. The extension of semialgebraic maps enjoys the expected behavior with respect to composition, direct and inverse images, injectivity, surjectivity, continuity, etc. \cite[\S5.1-3]{bcr}. By \cite[7.3.1]{bcr} the extension of semialgebraic functions to $F$ induces a well-defined $\R$-monomorphism ${\tt i}_{M,F}:{\mathcal S}(M)\hookrightarrow{\mathcal S}(M_{F}),\ f\mapsto f_{F}$. Composing it with the evaluation homomorphism ${\rm ev}_{M_F,\tt p}:{\mathcal S}(M_F)\to F,\ g\mapsto g({\tt p})$ for ${\tt p}\in M_{F}$, we get the $\R$-homomorphism 
$$
\psi_{\tt p}:={\rm ev}_{M_F,\tt p}\circ {\tt i}_{M,F}:{\mathcal S}(M)\to F,\ f\mapsto f_F({\tt p}).
$$ 
Denote the restriction of the linear projection onto the $i$th coordinate to $M$ with $\pi_i:M\to\R$. In \cite[Intr. Lem. 1, p.3]{fe3} it is proved that if ${\tt p}:=({\tt p}_1,\ldots,{\tt p}_m)\in M_F$, the $\R$-homomorphism $\psi_{\tt p}$ is the unique one satisfying $\pi_i\mapsto {\tt p}_i$ for $i=1,\ldots,m$.

\subsection{Formal paths}\label{picture} 

As usual $\R[[\t]]$ stands for the ring of formal power series in one variable with coefficients in $\R$ and $\R((\t))$ for its field of fractions. We say that a formal power series is \em algebraic \em if it is algebraic over the field of rational functions $\R(\t):=\qf(\R[\t])$. The subring (resp. subfield) of $\R[[\t]]$ (resp. $\R((\t))$) of all algebraic series is denoted with $\R[[\t]]_{\rm alg}$ (resp. $\R((\t))_{\rm alg}$). Given a formal power series $\xi\in\R((\t))$, we denote its order with $\omega(\xi)$ and the $k$-th power of the maximal ideal $(\t)$ of $\R[[\t]]$ with $(\t)^k$. We endow the previous rings with their respective unique orderings $\leq$ in which ${\tt t}$ is positive and infinitesimal with respect to $\R$. We denote the real closed field of Puiseux series with $F_1:=\R((\t^*))$ and the real closed field of algebraic Puiseux series with $F_0:=\R((\t^*))_{\rm alg}$. A \em formal path \em is a tuple $\alpha:=(\alpha_1,\ldots,\alpha_m)\in\R[[\t]]^m$. If $\alpha\in\R[[\t]]_{\rm alg}^m$, there exists $\veps>0$ such that the map $[0,\veps]\to\R^m,\ t\mapsto\alpha(t)$ is semialgebraic. Conversely, each semialgebraic map $\alpha:[0,1]\to\R^m$ defines an element $\alpha\in\R[[\t]]_{\rm alg}^m$. The elements of $\R[[\t]]_{\rm alg}^m$ are called \em semialgebraic paths\em. 

\subsection{Free maximal ideals associated with formal and semialgebraic paths}

Let $\alpha\in M_{F_1}$ be a formal path. Observe that $\alpha(0)\in\cl(M)$. By \ref{eoc} there exists a unique homomorphism $\psi_\alpha:{\mathcal S}(M)\to F_1$ such that $\psi_\alpha(\pi_i)=\alpha_i$. It holds: $\psi_\alpha({\mathcal S}^*(M))\subset\R[[{\tt t}^*]]$. 

Let $f\in{\mathcal S}^*(M)$ and $L>0$ be a constant such that $|f|<L$. Then $L-f>0$ and $f+L>0$. Pick $h_1,h_2\in{\mathcal S}^*(M)$ such that $h_1^2=L-f$ and $h_2^2=f+L$. Then
$$
L-\psi_\alpha(f)=\psi_\alpha(h_1^2)=\psi_\alpha(h_1)^2\geq0\quad\text{and}\quad\psi_\alpha(f)+L=\psi_\alpha(h_2^2)=\psi_\alpha(h_2)^2\geq0,
$$
so $|\psi_\alpha(f)|\leq L$. Consequently, $\psi_\alpha(f)\in\R[[\t^*]]$.

\subsubsection{Free maximal ideals associated with formal paths}\label{mifp} 
Consider the `evaluation' 
$$
{\rm ev}_0:\R[[\t^*]]\to\R,\ \zeta\mapsto\zeta(0)
$$ 
and the $\R$-epimorphism
$$
\varphi_\alpha:={\rm ev}_0\circ\psi_\alpha|_{{\mathcal S}^*(M)}:{\mathcal S}^*(M)\to\R,\ f\mapsto({\rm ev}_0\circ\psi_\alpha)(f)=\psi_\alpha(f)(0). 
$$
Then $\gtm_\alpha^*:=\ker(\varphi_\alpha)$ is a maximal ideal of ${\mathcal S}^*(M)$. As one can expect, $\gtm_{\alpha}^*=\gtm_{\alpha(0)}^*$ if $\alpha(0)\in M$ and $\gtm^*_\alpha$ is a free maximal ideal of ${\mathcal S}^*(M)$ when $\alpha(0)\in\cl(M)\setminus M$. In the latter case we call $\gtm_\alpha^*$ \em the free maximal ideal of ${\mathcal S}^*(M)$ associated with $\alpha$\em. We denote the collection of all free maximal ideals of ${\mathcal S}^*(M)$ associated with formal paths with $\widehat{\partial}M\subset\partial M$.

Let us find the maximal ideal $\gtm_\alpha$ of ${\mathcal S}(M)$ corresponding to $\gtm_\alpha^*$ via the homeomorphism $\Phi$ introduced in \ref{homeo}. We call $\gtm_\alpha$ \em the free maximal ideal of ${\mathcal S}(M)$ associated with $\alpha$\em. 

\begin{prop}\label{pmtalpha}
Let $\alpha\in M_{F_1}$ be a formal path such that $\alpha(0)\not\in M$. Then $\gtm_\alpha=\ker(\psi_\alpha)$ is the free maximal ideal of ${\mathcal S}(M)$ satisfying $\gtm_{\alpha}\cap{\mathcal S}^*(M)\subset\gtm_{\alpha}^*$. 
\end{prop}
\begin{proof}
It is straightforward to check that $\gtp_\alpha:=\ker(\psi_\alpha)$ is a prime $z$-ideal and $\gtp_\alpha\cap{\mathcal S}^*(M)\subset\ker(\varphi_\alpha)=\gtm^*_\alpha$. Let us show next that $\gtp_\alpha$ is a maximal ideal. Otherwise let $\gtq$ be a prime ideal of ${\mathcal S}(M)$ such that $\gtp_\alpha\subsetneq\gtq$ and choose $f\in\gtq\setminus\gtp_\alpha$. Taking $f/(1+|f|)$ instead of $f$, we may assume that $f$ is bounded on $M$. Denote $p:=\alpha(0)$.

As $\psi_{\alpha}(f)\in\R[[\t^*]]\setminus\{0\}$, we write $\psi_{\alpha}(f)(\t):=a\t^b+\cdots$ where $a\neq 0$, $b:=\omega(\psi_{\alpha}(f))\in\Q^+$ and $\|\alpha(\t)-p\|:=c\t^d+\cdots$ where $c\neq 0$ and $d:=\omega(\|\alpha(\t)-p\|)\in\Q^+$. Consider
$$
Z:=\Big\{x\in M:\,\frac{|a|}{2c^{b/d}}\|x-p\|^{b/d}\leq |f(x)|\Big\}
$$ 
and pick $g:=\dist(\cdot,Z)\in{\mathcal S}(M)$, which satisfies $Z(g)=Z$. Observe 
$$
\frac{|a|}{2c^{b/d}}\|\alpha(\t)-p\|^{b/d}=\frac{|a|}{2}\t^b+\cdots\qquad\text{and}\qquad\psi_{\alpha}(|f|)(\t)=|a|\t^b+\cdots,
$$
so $\alpha\in Z_{F_1}$. Consequently, $\psi_\alpha(g)=g_{F_1}(\alpha)=0$ or equivalently $g\in\gtp_\alpha\subset\gtq$.

As $p\not\in M$, the zero set of $h:=f^2+g^2\in\gtq$ is the empty-set, so $h\in\gtq$ is a unit in ${\mathcal S}(M)$, which is a contradiction. Thus, $\gtp_\alpha$ is the maximal ideal of ${\mathcal S}(M)$ satisfying $\gtp_\alpha\cap{\mathcal S}^*(M)\subset\gtm^*_\alpha$. 
\end{proof}
\begin{remark}\label{pmtalphar}
The reader can check that \em if $\alpha(0)\in M$, the ideal $\gtp_\alpha=\ker(\psi_\alpha)$ is a prime $z$-ideal of ${\mathcal S}(M)$ of coheight $1$ contained in $\gtm_{\alpha(0)}$\em. 
\end{remark}

\begin{cor}\label{ppp}
Let $\alpha\in M_{F_1}$ be a formal path such that $\alpha(0)\in\cl(M)\setminus M$. Then there does not exist any prime ideal between $\gtm_\alpha\cap{\mathcal S}^*(M)$ and $\gtm_\alpha^*$.
\end{cor}
\begin{proof}
Let $\gtp_0:=\gtm_\alpha\cap{\mathcal S}^*(M)\subsetneq\cdots\subsetneq\gtp_r=\gtm_\alpha^*$ be the collection of all prime ideals of ${\mathcal S}^*(M)$ containing $\gtp_0$. By \cite[1.A.2]{fg2} there exist a semialgebraic compactification $(X,{\tt j})$ of $M$ and a chain of prime ideals $\gtq_0\subsetneq\cdots\subsetneq\gtq_r$ of ${\mathcal S}(X)$ such that $\gtq_i=\gtp_i\cap{\mathcal S}(X)$. Assume $X\subset\R^m$ and notice that $\gtq_0=\ker(\psi_{{\tt j}\circ\alpha})$ where ${\tt j}\circ\alpha\in\R[[\t^*]]^m$. After reparameterizing $\alpha$ if necessary, we may assume ${\tt j}\circ\alpha\in\R[[\t]]^m$. Proceeding similarly to the proof of Proposition \ref{pmtalpha} one finds a semialgebraic function $h\in\gtq_1$ such that $Z(h)=\{({\tt j}\circ\alpha)(0)\}$, so $\gtq_1$ is the maximal ideal of ${\mathcal S}(X)$ associated with the point $({\tt j}\circ\alpha)(0)$ and $r=1$, as required.
\end{proof}

\subsubsection{Free maximal ideals associated with semialgebraic paths}\label{misp}
The collection of all free maximal ideals $\gtm^*_\alpha$ of ${\mathcal S}^*(M)$ corresponding to semialgebraic paths $\alpha\in M_{F_0}$ is denoted with $\widetilde{\partial} M$. We have $\widetilde{\partial} M\subset\widehat{\partial}M\subset\partial M$ and in general both inclusions are strict and the differences are `large' (see \ref{diferencias}). The uniqueness (see \ref{eoc}) of the homomorphism $\psi_\alpha$ guarantees that if $\alpha\in M_{F_0}$ is a semialgebraic path, the $\R$-homomorphism $\psi_\alpha:{\mathcal S}(M)\to F_0$ is defined by $f\mapsto f\circ\alpha$. If $\alpha\in M_{F_0}$ and $\alpha(0)\in\cl(M)\setminus M$, then
\begin{align*}
\gtm_\alpha^*&=\{f\in{\mathcal S}^*(M):\,\lim_{t\to0^+}(f\circ\alpha)(t)=0\},\\
\gtm_\alpha&=\{f\in{\mathcal S}(M):\,\exists\,\veps>0\text{ such that }(f\circ\alpha)|_{(0,\,\veps]}=0\}.
\end{align*}
\begin{remark}\label{etac}
The reader can check that \em the prime $z$-ideals of ${\mathcal S}(M)$ whose semialgebraic depth is equal to $1$ are the prime ideals $\gtp_\alpha:=\ker(\psi_\alpha)$ where $\alpha\in M_{F_0}$ is a semialgebraic path\em. In addition, $\qf({\mathcal S}^*(M)/\gtp_\alpha)=\R((\t^*))_{\rm alg}$.
\end{remark}

\subsection{Set of endpoints of a semialgebraic set}
We finish this section with some properties of the sets $\eta(M)$ and $\eta(\betaa M)$ of endpoints of a semialgebraic set $M$ and its semialgebraic Stone-C\v{e}ch compactification $\betaa M$.

\begin{cor}\label{eum}
We have:
\begin{itemize}
\item[(i)] $\eta(\betaa M)=\eta(M)\cup(\cl_{\betaa M}(\ol{M}_{\leq1})\setminus\ol{M}_{\leq1})$ is a finite set.
\item[(ii)] For each point $p\in \eta(M)$ the maximal ideal $\gtm_p^*$ of ${\mathcal S}^*(M)$ contains properly only one prime ideal of ${\mathcal S}^*(M)$.
\item[(iii)] $M\setminus\eta(M)$ is the set of closed branching points of $\Speca(M)$.
\end{itemize}
\end{cor}
\begin{proof}
(i) As $M=M_{\geq2}\cup\ol{M}_{\leq1}$, it holds $\betaa M=\cl_{\betaa M}(M_{\geq2})\cup\cl_{\betaa M}(\ol{M}_{\leq1})$. Notice that no point of $\cl_{\betaa M}(M_{\geq2})$ has a neighborhood homeomorphic to $[0,1)$. Now, one shows (following the proof of \cite[4.19]{fg5}) that the set of endpoints of $\betaa M$ equals $(\cl_{\betaa M}(\ol{M}_{\leq1})\setminus\ol{M}_{\leq1})\cup\eta(M)$ and $(\cl_{\betaa M}(\ol{M}_{\leq1})\setminus\ol{M}_{\leq1})$ is a finite set. In addition, $\eta(M)$ is a semialgebraic set of dimension $0$, so it is also a finite set.

(ii) Let $Z\subset M$ be a compact semialgebraic neighborhood of $p$ equipped with a (semialgebraic) homeomorphism $Z\to[0,1]$ that maps $p$ onto $0$. Let $T:=\cl(M\setminus Z)\cap M$ and note that $p\not\in T$ and $M=T\cup Z$. Then $\Speca(M)=\cl_{\Speca(M)}(T)\cup\cl_{\Speca(M)}(Z)$. We have $p\equiv\gtm_p^*\in\Speca(M)\setminus\cl_{\Speca(M)}(T)=\cl_{\Speca(M)}(Z)\setminus\cl_{\Speca(M)}(T)$. 

Next, $\cl_{\Speca(M)}(Z)$ is by \ref{closedspec} homeomorphic to $\Speca(Z)\cong\Speca([0,1])$ and $\gtm^*_p$ is mapped to the maximal ideal $\gtm_0^*$ of $\Speca([0,1])$. As $I:=[0,1]$ is locally compact, we know by \ref{ht} that $0=\dgt_I(\gtm_0^*)<\dgt_I(\gtp)\leq1$ for each prime ideal $\gtp$ of ${\mathcal S}(I)$ (properly) contained in $\gtm_0^*$. Thus, $\dgt_I(\gtp)=1$ and $\gtp$ is by \cite[4.5]{fe1} a minimal prime ideal contained in $\gtm_0^*$. As $0$ is an endpoint of $[0,1]$, we deduce that $\gtp$ is the unique prime ideal (properly) contained in $\gtm_0^*$. Therefore $\gtm_p^*$ contains only one prime ideal of ${\mathcal S}(M)$.

(iii) Fix $a\in M\setminus\eta(M)$. By the curve selection lemma \cite[2.5.5]{bcr} there exist two semialgebraic paths $\alpha_1,\alpha_2:[0,1]\to\R^m$ such that $\alpha_i(0)=a$, $\alpha_i((0,1])\subset M$ and $\alpha_1((0,1])\cap\alpha_2((0,1])=\varnothing$. Let ${\mathcal W}:=\{g\in{\mathcal S}^*(M):\ Z(g)=\varnothing\}$. As ${\mathcal S}(M)={\mathcal S}^*(M)_{{\mathcal W}}$ and $\gtm_a\cap{\mathcal S}^*(M)=\gtm_a^*$, there exists a one-to-one correspondence that preserves inclusions between the prime ideals of ${\mathcal S}(M)$ contained in $\gtm_a$ and those of ${\mathcal S}^*(M)$ contained in $\gtm_a^*$. Consider the prime $z$-ideals
$$
\gtp_{\alpha_i}:=\{f\in{\mathcal S}(M):\,\exists\,\veps>0\,|\ (f\circ\alpha_i)|_{(0,\,\veps)}=0\}.
$$
By Remark \ref{etac} $\dgt_M(\gtp_{\alpha_i})=1$ while $\dgt_M(\gtm_a)=0$. Thus, $\gtp_{\alpha_i}$ has coheight $1$ in $\gtm_a$ by \ref{ht}, so $\gtp_{\alpha_i}\cap{\mathcal S}^*(M)$ has coheight $1$ in $\gtm_a^*$. Consequently, $\gtm_a^*$ is a closed branching point of $\Speca(M)$. 

Conversely, let $\gtm^*\in\Speca(M)$ be a closed branching point. By (ii) $\gtm^*\not\in\eta(M)$, so we have to check $\gtm^*\in M$. Suppose by contradiction $\gtm^*\in\Speca(M)\setminus M$. By \ref{freefixed} the unique maximal ideal $\gtm$ of ${\mathcal S}(M)$ with $\gtm\cap{\mathcal S}^*(M)\subset\gtm^*$ satisfies $\gtm\cap{\mathcal S}^*(M)\subsetneq\gtm^*$. By \cite[5.2(i)]{fe1} the subchain of prime ideals of ${\mathcal S}^*(M)$ containing $\gtm\cap{\mathcal S}^*(M)$ is the same for any non-refinable chain of prime ideals in ${\mathcal S}^*(M)$ ending at $\gtm^*$. As $\gtm\cap{\mathcal S}^*(M)\subsetneq\gtm^*$, the maximal ideal $\gtm^*$ only contains one prime ideal of coheight $1$, which is a contradiction. 
\end{proof}

\section{Topological properties of the remainder}\label{s4}

We study the topological properties of the remainder announced in the Introduction.

\subsection{Density of $\widetilde{\partial}M$ in $\partial M$.}\label{ppoints}
We prove first that $\widetilde{\partial} M$ is dense in $\partial M$. 

\begin{lem}\label{deltam}
\em (i) \em Let $f_i\in{\mathcal S}^*(M)$ and $\widehat{f_i}:\betaa M\to \R$ be the unique continuous extension of $f_i$ to $\betaa M$ for $i=1,\ldots,r$. Then $(\widehat{f_1},\ldots,\widehat{f_r})(\widetilde{\partial}M)=(\widehat{f_1},\ldots,\widehat{f_r})(\partial M)$.

\em (ii) \em A function $f\in{\mathcal S}^*(M)$ is a unit if and only if $Z(f)=\varnothing$ and $f\not\in\gtm_\alpha^*$ for all $\gtm_\alpha^*\in\widetilde{\partial}M$. 

\em (iii) \em The set $\widetilde{\partial}M$ is dense in $\partial M$.
\end{lem}
\begin{proof} 
(i) Assume $M$ is bounded, so $\cl(M)$ is a semialgebraic compactification of $M$. Thus, there exists by \cite[4.6]{fg5} a surjective continuous map $\rho:\betaa M\to\cl(M)$ that is the identity on $M$. Fix $\gtm^*\in\partial M$ and observe that by \cite[4.3(i)]{fg5} $p:=\rho(\gtm^*)\in\cl(M)\setminus M$. Consider the proper map $\Psi:=(\rho,\widehat{f}):\betaa M\to\R^{m+r}$ where we abbreviate $f:=(f_1,\ldots,f_r)$ and $\widehat{f}:=(\widehat{f_1},\ldots,\widehat{f_r})$ and denote $a:=\widehat{f}(\gtm^*)$. Clearly, $\Psi(M)$ is the graph $\Gamma$ of $f$ and since $\Psi$ is proper, 
$$
\im\Psi=\Psi(\cl_{\betaa M}(M))=\cl_{\R^{m+r}}(\Psi(M))=\cl_{\R^{m+r}}(\Gamma).
$$
Again by \cite[4.3(i)]{fg5} $q:=\Psi(\gtm^*)=(\rho(\gtm^*),\widehat{f}(\gtm^*))=(p,a)\in\cl_{\R^{m+r}}(\Gamma)\setminus\Gamma\subset\R^m\times\R^r$.

By the curve selection lemma there exist semialgebraic paths $\alpha:[0,1]\to\R^m$ and $\mu:[0,1]\to\R^r$ such that $\alpha((0,1])\subset M$, $\mu|_{(0,\,1]}=(f\circ\alpha)|_{(0,\,1]}$ and $(\alpha(0),\mu(0))=q$. Consequently,
$$
\widehat{f}(\gtm^*)=a=\mu(0)=\lim_{t\to0^+}\mu(t)=\lim_{t\to0^+}(f\circ\alpha)(t)=\widehat{f}(\gtm_\alpha^*)
$$
where $\gtm_\alpha^*\in \widetilde{\partial}M$ because $\lim_{t\to0^+}\alpha(t)=p\not\in M$. 

(ii) Observe that $f\in{\mathcal S}^*(M)$ is a unit if and only if $0\notin\widehat{f}(\betaa M)=f(M)\cup\widehat{f}(\partial M)=f(M)\cup\widehat{f}(\widetilde{\partial}M)$ (see assertion (i)), which proves the statement.

(iii) We have to check that for every $f\in {\mathcal S}^*(M)$ such that $\di_{\betaa M}(f)\not \subset M$ the intersection $\di_{\betaa M}(f)\cap\widetilde{\partial}M$ is non-empty. Otherwise, $\widetilde{\partial}M\subset\ceros_{\betaa M}(f)$ and by part (i) we obtain $\{0\}=\widehat{f}(\widetilde{\partial}M)=\widehat{f}(\partial M)$ or equivalently $\di_{\betaa M}(f)\subset M$, which is a contradiction.
\end{proof}

\begin{remark}
As a consequence of the previous result one can show the following well-known fact \cite[2.1, 2.2]{b}: \em A surjective semialgebraic map $g:N\to M$ is proper if and only if $\betaa\,g(\partial N)=\partial M$\em.

The `if' part is clear. For the converse assume $M$ bounded. Denote $\widehat{g}:=(\widehat{g_1},\ldots,\widehat{g_m}):\betaa N\to\R^m$ where $\widehat{g_i}$ is the (unique continuous) extension of the component $g_i$ of $g$ to $\betaa N$. Suppose there exists a point $\gtn^*\in\partial N$ such that $p:=\betaa\,g(\gtn^*)\in M$. By Lemma \ref{deltam}(i) there exists $\gtn^*_\alpha\in\widetilde{\partial}N$ such that $\widehat{g}(\gtn^*_{\alpha})=p$, so $\betaa\,g(\gtn^*_\alpha)=\lim_{t\to0^+}(g\circ\alpha)(t)=\widehat{g}(\gtn^*_{\alpha})=p\in M$. As $g$ is proper, $\alpha(0)=\lim_{t\to0^+}\alpha(t)$ belongs to $N$, which contradicts the fact that $\gtn^*_\alpha\in\widetilde{\partial}N$. Consequently, $\betaa\,g(\partial N)\subset\partial M$. The converse inclusion follows because $\betaa\,g$ is surjective.
\end{remark}

\subsection{Differences between the sets $\widetilde{\partial}M$, $\widehat{\partial}M$ and $\partial M$}\label{diferencias}

We prove next that the non-empty differences $\partial M\setminus\widehat{\partial}M$ and $\widehat{\partial}M\setminus\widetilde{\partial}M$ are respectively dense in $\partial M$ and $\widehat{\partial}M$ under mild conditions.

\begin{thm}\label{diferencias-s}
Assume that $M=M_{\geq2}$ is not compact. Then $\partial M\setminus\widehat{\partial}M$ is dense in $\partial M$ and $\widehat{\partial}M\setminus\widetilde{\partial}M$ is dense in $\widehat{\partial}M$.
\end{thm}

We begin with some preliminary results.

\begin{lem}\label{pppoints0}
Assume that $M$ is bounded. Then $\partial M\setminus\widehat{\partial}M\neq\varnothing$ if and only if $M_{\geq2}$ is not compact. In addition, if $M_{\geq2}$ is compact, then $\widetilde{\partial} M=\partial M$ is a finite set.
\end{lem}
\begin{proof}
Suppose first that $M_{\geq2}$ is compact. The finiteness of $\partial M$ follows from \cite[5.17]{fg5}, so by Lemma \ref{deltam}(iii) $\widetilde{\partial}M=\partial M$. Conversely, suppose that $M_{\geq2}$ is not compact. By \cite[7.1(i)]{fe1} there exists a point $p\in\cl(M_{\geq2})\setminus(\cl(\rho_1(M_{\geq2}))\cup M_{\geq2})$. Notice that $\rho_1(M)=\rho_1(M_{\geq2})$. In addition, $M_{\geq2}$ is closed in $M$, so $p\in\cl(M)\setminus(\cl(\rho_1(M))\cup M)$ and $\dim_p(\cl(M))\geq2$. By \cite[7.1(ii)]{fe1} there exists a maximal ideal $\gtm^*$ of ${\mathcal S}^*(M)$ of height $\geq 2$ such that $\hgt(\gtm)=0$. This implies by Corollary \ref{ppp} that $\gtm^*\in\partial M\setminus\widehat{\partial}M$, as required.
\end{proof}

\begin{lem}[Behavior of the operators $\widetilde{\partial}$ and $\widehat{\partial}$]\label{delrel}
Assume that $M$ is bounded and let $Y\subset M$ be a closed semialgebraic subset of $M$. As the semialgebraic sets $Y$, $M_{\geq2}$ and $\ol{M}_{\leq1}$ are closed in $M$, we identify $\cl_{\betaa M}(Y)\equiv\betaa Y$, $\cl_{\betaa M}(M_{\geq2})\equiv\betaa M_{\geq2}$ and $\cl_{\betaa M}(\ol{M}_{\leq1})\equiv\betaa\ol{M}_{\leq1}$. Then
\begin{itemize}
\item[(i)] $\widetilde{\partial} Y=\widetilde{\partial}M\cap\partial Y$ and $\widehat{\partial} Y=\widehat{\partial}M\cap\partial Y$.
\item[(ii)] $\partial M=\partial M_{\geq2}\sqcup\partial\ol{M}_{\leq1}$. 
\item[(iii)] $\widetilde{\partial} M=\widetilde{\partial}M_{\geq2}\sqcup\widetilde{\partial}\,\ol{M}_{\leq1}$ and $\widehat{\partial} M=\widehat{\partial}M_{\geq2}\sqcup\widehat{\partial}\,\ol{M}_{\leq1}$.
\end{itemize}
\end{lem}
\begin{proof}
(i) Let us check first $\widehat{\partial} Y=\widehat{\partial}M\cap\partial Y$. For the non-obvious inclusion let $\gtm^*_\alpha\in\widehat{\partial}M\cap\partial Y$. Suppose by contradiction $\gtm^*_\alpha\not\in\widehat{\partial} Y$, that is, $\alpha\not\in Y_{F_1}$. Thus, $\alpha\in(M\setminus Y)_{F_1}$ and there exists $g\in{\mathcal S}^*(M)$ such that $\alpha\in(D(g))_{F_1}\subset(M\setminus Y)_{F_1}$. In particular, $g|_{Y}\equiv 0$ and $\psi_{\alpha}(g)\neq0$. Write $\psi_{\alpha}(g):=a\t^p+\cdots$ for some $a\neq 0$ and a non-negative rational number $p$ and $\|\alpha(\t)-\alpha(0)\|:=b\t^q+\cdots$ for some $b\neq 0$ and a positive $q\in\Q$. Recall that $\alpha(0)\not\in M$ because $\gtm_\alpha^*\in\partial Y$ and consider the bounded semialgebraic function
$$
f:M\to\R,\ x\mapsto\frac{g^2(x)}{g^2(x)+\|x-\alpha(0)\|^{2(p/q)+1}},
$$
which vanishes identically on $Y$ and satisfies $\psi_{\alpha}(f)(0)=1$. Thus, $f\in\ker\phi\setminus\gtm^*_\alpha$ where $\phi:{\mathcal S}^*(M)\to{\mathcal S}^*(Y),\,h\mapsto h|_Y$. This contradicts the fact that $\gtm^*_\alpha\in\partial Y\equiv\cl_{\betaa M}(Y)\setminus Y$ because $\cl_{\betaa M}(Y)$ is  by \cite[6.3]{fg3} the collection of those maximal ideals of ${\mathcal S}^*(M)$ containing $\ker\phi$. 

The first equality in (i) follows from the equality already proved above because the semialgebraic character of a formal path does not depend on the semialgebraic set where it is considered. Statement (ii) follows by considering the connected components of $\partial M$ and noticing that the union of the ones that are singletons belongs to $\partial\ol{M}_{\leq1}$. Statement (iii) follows from (i) and (ii).
\end{proof}

\begin{remark}
The assumption $M=M_{\geq2}$ in Theorem \ref{diferencias-s} is not restrictive. As $\dim(\ol{M}_{\leq1})\leq 1$, we deduce from Lemma \ref{pppoints0} that $\widetilde{\partial}\,\ol{M}_{\leq1}=\widehat{\partial}\,\ol{M}_{\leq1}=\partial\ol{M}_{\leq1}$. By Lemma \ref{delrel} we obtain
$$
\partial M\setminus\widehat{\partial}M=\partial M_{\geq2}\setminus\widehat{\partial}M_{\geq2}\quad\text{and}\quad\widehat{\partial}M\setminus \widetilde{\partial}M=\widehat{\partial}M_{\geq2}\setminus \widetilde{\partial}M_{\geq2},
$$
so Theorem \ref{diferencias-s} is conclusive.
\end{remark}

\begin{proof}[Proof of Theorem \em\ref{diferencias-s}]
We prove first the following:\setcounter{paragraph}{0} 

\paragraph{}\label{redtrianT}
\em The sets $\partial T\setminus\widehat{\partial}T$ and $\widehat{\partial}T\setminus \widetilde{\partial}T$ are not empty for the punctured triangle 
$$
T:=\{(x,y)\in\R^2:\,0\leq y\leq x\leq1\}\setminus\{(0,0)\}.
$$\em

By Lemma \ref{pppoints0} we obtain $\partial T\setminus\widehat{\partial}T\neq\varnothing$. In order to prove $\widehat{\partial}T\setminus\widetilde{\partial}T\neq\varnothing$ choose the formal series $\alpha_1(\t)=\t$ and $\alpha_2(\t)=\sum_{n\geq2}n!\t^n\in\R[[\t]]\setminus\R[[\t]]_{\rm alg}$ and the formal path $\alpha:=(\alpha_1,\alpha_2)\in\R[[\t]]^2$. Note that $\alpha\in T_{F_1}$ and let us show $\gtm_\alpha^*\in\widehat{\partial}T\setminus\widetilde{\partial}T$. 

For each $k\geq 2$ consider the semialgebraic function $f_k\in{\mathcal S}^*(T)$ given by the formula
$$
f_k(x,y):=\frac{(y-p_k(x))^2}{(y-p_k(x))^2+x^{2k}}\quad\text{where $p_k(x):=\sum_{n=2}^{k}n!x^n$}.
$$
We have $\psi_{\alpha}(f_k)(0)=0$, so $f_k\in\gtm^*_\alpha$ for all $k\geq2$ (see \ref{mifp}). Suppose now that $\gtm_\alpha=\gtm_\mu$ for some $\mu\in\R[[\t]]_{\rm alg}^2$ with $\mu\in T_{F_1}$. To obtain a contradiction, it is enough to check that $f_k\not\in\gtm_\mu$ for some $k\geq2$. Without loss of generality and after reparameterizing $\mu$, we may assume $\mu(\t)=(\t^j,\mu_2(\t))$ for some integer $j\geq1$ and some analytic series $\mu_2(\t)\in\R[[\t]]_{\rm alg}$ whose order is $\geq j$. As the series $\alpha_2(\t^j)$ is not analytic, $\alpha_2(\t^j)-\mu_2(\t)\neq 0$ and its order is $p\geq 1$. Thus, $\psi_\mu(f_k)(0)\neq 0$ for $k=p+1$, so $f_k\not\in\gtm_\mu$, as claimed.

\paragraph{} We show the statement under the assumption that $M$ is bounded. 

Let $f\in {\mathcal S}^*(M)$ be such that $\di_{\betaa M}(f)$ meets $\partial M$ and $\widehat{f}:\betaa M\to\R$ be the (unique) continuous extension of $f$ to $\betaa M$. As $\widetilde{\partial}M$ is dense in $\partial M$, there exists $\gtm_\alpha^*\in\widetilde{\partial}M\cap\di_{\betaa M}(f)$. Write $c:=\widehat{f}(\gtm^*_\alpha)\neq0$ and assume $c>0$. Thus, $\gtm^*_\alpha\in{\mathcal D}_{\betaa M}(f-\tfrac{c}{2}+|f-\tfrac{c}{2}|)$. Substituting $M$ by $\gr(f)$, we may assume that $f$ can be extended continuously to $X:=\cl(M)$. Denote such extension with $f$. By \cite[4.3\&4.6]{fg5} there exists a continuous surjective map $\rho:\betaa M\to X$ that is the identity on $M$ and $\rho(\partial M)=X\setminus M$. In addition, $\widehat{f}=f\circ\rho$ and $p:=\rho(\gtm_\alpha^*)\in X\setminus M$ satisfies $f(p)=c$. 

Define $Y_0:=\{p\}$, $Y_1:=\{f-\tfrac{c}{2}>0\}$ and $Y_2:=M\cap Y_1$. By \cite[9.2.1]{bcr} there exists a finite simplicial complex $K$ and a semialgebraic homeomorphism $\Phi:|K|\to X$ such that each semialgebraic set $Y_j$ is the union of some $\Phi(\sigma^0)$ where each $\sigma^0$ is the open simplex associated with a simplex $\sigma\in K$. We identify $X$ with $|K|$ and choose a simplex $\tau$ of $K$ of dimension $\geq2$ that has $p$ as a vertex and whose associated open simplex $\tau^0$ is contained in $Y_2$. Let $p_1,p_2\in\tau^0$ be two points that are not colinear with $p$. For the closed triangle $T_1$ with vertices $p,p_1,p_2$ it holds that $T_1\setminus\{p\}\subset\tau\subset Y_2$ is a closed semialgebraic subset of $M$. In addition, $\partial T_1\subset\di_{\betaa M}(f)$. Thus, the differences $\partial T_1\setminus\widehat{\partial}T_1$ and $\widehat{\partial}T_1\setminus\widetilde{\partial}T_1$ are by \ref{redtrianT} non-empty and the open set $\di_{\betaa M}(f)$ meets the differences $\partial M\setminus\widehat{\partial}M$ and $\widehat{\partial}M\setminus \widetilde{\partial}M$ by Lemma \ref{delrel}, as required.
\end{proof}

\subsection{Points of the maximal spectrum with countable basis of neighborhoods}\label{furapp}\setcounter{paragraph}{0}

We show next that all points of $M$ have countable basis of neighborhoods in $\betaa M$. This is trivially true for the points of $M_{\lc}$ as $M_{\lc}$ is open in $\betaa M$. The points of $\rho_1(M)$ require a careful analysis.

\begin{lem}\label{fc}
Let $f\in{\mathcal S^*(M)}$ and $\widehat{f}:\betaa M\to\R$ be its unique continuous extension to $\betaa M$. Let $\gtm\in\betaa M$ be such that $c:=\widehat{f}(\gtm)>0$. Then $\cl_{\betaa M}(f^{-1}((\tfrac{c}{2},+\infty)))=\cl_{\betaa M}(\widehat{f}^{-1}(\tfrac{c}{2},+\infty))$.
\end{lem}
\begin{proof}
It is enough to check $\widehat{f}^{-1}((\tfrac{c}{2},+\infty))\subset\cl_{\betaa M}(f^{-1}(\tfrac{c}{2},+\infty))$. Fix $\gtn\in\widehat{f}^{-1}((\tfrac{c}{2},+\infty))$ and let $V$ be a neighborhood of $\gtn$ in $\betaa M$. Then $V\cap\widehat{f}^{-1}((\tfrac{c}{2},+\infty))$ is also a neighborhood of $\gtn$ in $\betaa M$. As $M$ is dense in $\betaa M$,
$$
V\cap f^{-1}((\tfrac{c}{2},+\infty))=V\cap\widehat{f}^{-1}((\tfrac{c}{2},+\infty))\cap M\neq\varnothing.
$$
Thus, $\gtn\in\cl_{\betaa M}(f^{-1}((\tfrac{c}{2},+\infty)))$, as required.
\end{proof}

\begin{prop}\label{cn}
Let $p\in M$ and $\{U_k\}_k$ be a countable basis of neighborhoods of $p$ in $M$. Then $\{\cl_{\betaa M}(U_k)\}_k$ is a countable basis of neighborhoods of $p$ in $\betaa M$.
\end{prop}
\begin{proof}
Let $W$ be an open neighborhood of $p$ in $\betaa M$. Then there exists $f\in{\mathcal S}^*(M)$ such that $p\in{\mathcal D}_{\betaa M}(f)\subset W$. Let $\widehat{f}:\betaa M\to\R$ be the unique continuous extension of $f$ to $\betaa M$. We may assume $\widehat{f}(p)=c>0$ and observe that $f^{-1}((\tfrac{c}{2},+\infty))$ is an open neighborhood of $p$ in $M$. Thus, there exists $k\geq 1$ such that $p\in U_k\subset f^{-1}((\tfrac{c}{2},+\infty))$. Therefore
$$
\cl_{\betaa M}(U_k)\subset\cl_{\betaa M}(f^{-1}((\tfrac{c}{2},+\infty)))\subset\cl_{\betaa M}(\widehat{f}^{-1}((\tfrac{c}{2},+\infty)))\subset\widehat{f}^{-1}([\tfrac{c}{2},+\infty))\subset W.
$$

To finish, let us see that each set $\cl_{\betaa M}(U_k)$ is a neighborhood of $p$ in $\betaa M$. Let $W_k$ be a neighborhood of $p$ in $\betaa M$ such that $U_k=W_k\cap M$. Let $g\in{\mathcal S}^*(M)$ be such that $p\in{\mathcal D}_{\betaa M}(g)\subset W_k$. Then $p\in D(g)\subset U_k$. We may assume $r=g(p)>0$, so $p\in\widehat{g}^{-1}((r/2,+\infty))\subset W_k$. Thus, $p\in g^{-1}((r/2,+\infty))\subset W_k\cap M=U_k $ and by Lemma \ref{fc} 
$$
p\in\widehat{g}^{-1}((r/2,+\infty))\subset\cl_{\betaa M}(\widehat{g}^{-1}((r/2,+\infty)))=\cl_{\betaa M}(g^{-1}((r/2,+\infty)))\subset\cl_{\betaa M}(U_k).
$$
Consequently, $\cl_{\betaa M}(U_k)$ is a neighborhood of $p$ in $\betaa M$, as required.
\end{proof}

We prove next that there exist a lot of points in $\partial M$ that have a countable basis of neighborhoods in $\betaa M$. We denote the open ball of $\R^m$ with center $x$ and radius $\veps>0$ with $\Bb(x,\veps)$. 

\begin{thm}\label{gpn3}
Each point of $\widehat{\partial}M$ has a countable basis of neighborhoods in $\betaa M$.
\end{thm}
\begin{proof}
Let $\alpha:=(\alpha_1,\ldots,\alpha_m)\in M_{F_1}$ be a formal path such that $\alpha(0)\in\cl(M)\setminus M$. Our aim is to construct a countable basis of neighborhoods for $\gtm_\alpha^*$ in $\betaa M$. 

\paragraph{}\label{piccolino1}
We may assume: \em $M\subset\{x_1>0\}$, $\alpha(0)=0$ and $\alpha_1(\t)=\t$\em. 

After a change of coordinates in $\R^m$ we may assume $\alpha(0)=0$ and that $\alpha_1$ is not a constant. Considering the embedding of $\R^m$ in $\R^{m+1}$ given by 
$$
(x_1,\ldots,x_m)\mapsto(x_1^2+\cdots+x_m^2,x_1,\ldots,x_m)=(y_1,\ldots,y_{m+1}),
$$ 
we can suppose $M\subset\{y_1>0\}$. After reparameterizing $\alpha$, we assume $\alpha_1(\t)=\t^p$ for some integer $p\geq1$. This in combination with the new change of coordinates 
$$
h:(0,+\infty)\times\R^{m}\to(0,+\infty)\times\R^{m},\ (y_1,y_2,\ldots,y_{m+1})\mapsto(\sqrt[p]{y_1},y_2,\ldots,y_{m+1})
$$
allows us to suppose $\alpha_1(\t)=\t$.

\paragraph{}\label{piccolino2}
For each integer $\ell\geq 1$ consider polynomials $\gamma_{2\ell},\ldots,\gamma_{m\ell}\in\R[\t]$ such that $\alpha_j-\gamma_{j\ell}\in(\t)^{\ell+1}\subset\R[[\t]]$ and let $L_\ell>0$ be such that $|\gamma_{j\ell}(t)|<L_\ell$ for $|t|\leq1$ and $j=2,\ldots,m$. Denote $\gamma_\ell(\t):=(\t,\gamma_{2\ell}(\t),\ldots,\gamma_{m\ell}(\t))$ and consider the family of semialgebraic functions on $M$
$$
\begin{cases}
f_\ell(x):=x_1^{2\ell+2}-\|x-\gamma_\ell(x_1)\|^2=x_1^{2\ell+2}-\sum_{j=2}^m(x_j-\gamma_{j\ell}(x_1))^2&\text{for}\ \ell\geq 1,\\[4pt] 
h_k(x):=\frac{1}{k^2}-x_1^2&\text{for}\ k\geq 1
\end{cases}
$$
and the family of open subsets $U_{\ell,k}:={\mathcal D}_{\betaa M}(f_\ell+|f_\ell|,h_k+|h_k|)$ of $\betaa M$. Note that $\gtm_\alpha^*\in U_{\ell,k}$ for all $\ell,k\geq 1$. 

\paragraph{}\label{piccolino20}
Our goal is to see: \em $\{U_{\ell,k}\}_{\ell,k}$ is a basis of neighborhoods of $\gtm_\alpha^*$ in $\betaa M$\em. 

Fix $g\in{\mathcal S}^*(M)$ such that $\gtm^*_\alpha\in{\mathcal D}_{\betaa M}(g)$ and assume $\widehat{g}(\gtm^*_\alpha)=c>0$. We write $V:=g^{-1}((\tfrac{c}{2},+\infty))$. Notice that $\alpha\in V_{F_1}$ and choose polynomials $g_1,\ldots,g_r\in\R[\x]$ such that $V_1:=\{g_1>0,\ldots,g_r>0\}$ satisfies $\alpha\in(V_1\cap M)_{F_1}\subset V_{F_1}$. 

Consider the new variables $\s$, $\y:=(\y_1,\ldots,\y_m)$ and $\z:=(\z_1,\ldots,\z_m)$, write $\x=\y+\s\z$ and 
$$
g_i(\x)=g_i(\y+\s\z)=g_i(\y)+\s H_i(\s,\y,\z)
$$
where $H_i(\s,\y,\z):=\sum_{j=1}^{s_i-1}h_{ij}(\y,\z)\s^j$ for some polynomials $h_{i1},\ldots,h_{i,s_i-1}\in\R[\y,\z]$. Let $C_\ell>0$ be a large enough real number such that 
$$
\text{$|H_i(s,y,z)|<C_\ell$\quad for $|s|\leq1,|z_j|\leq1,|y_1|\leq1,|y_2|\leq L_\ell,\ldots,|y_m|\leq L_\ell$},
$$
$j=1,\ldots,m$ and $i=1,\ldots,r$. 

\paragraph{}\label{piccolo}
As the positivity of a finite family of polynomials on a formal path depends only on finitely many terms of its components, \em there exists $\ell_0\geq1$ such that for all $\ell\geq \ell_0$ every formal path $\eta\in\R[[\t]]^m$ with $\|\eta(\t)-\alpha(\t^p)\|^2\in(\t)^{2\ell p}$ for some $p\geq1$ satisfies $g_i(\eta(\t))>0$ for $i=1,\ldots,r$\em. In particular: \em If $\ell\geq\ell_0$, each series $g_i(\gamma_{\ell}(\t))$ is positive\em. 

\paragraph{}\label{piccolo2}
Denote $\ell:=1+\max\{\ell_0,\omega(g_i(\alpha(\t))):\ i=1,\ldots,r\}$ and choose $k_0\geq 1$ such that $g_i(\gamma_{\ell}(t))>0$ for $i=1,\ldots,r$ if $0<t<1/k_0$. Since $\ell>\omega(g_i(\alpha(\t)))$ for $\ i=1,\ldots,r$, there exists $k\geq k_0$ such that $
g_i(\gamma_{\ell}(t))-t^{\ell+1}C_\ell>0$ for $0<t\leq1/k$ and $i=1,\ldots,r$. 

\paragraph{}\label{piccolo3}
For our purposes it is enough to check: $U_{\ell,k}\subset\widehat{g}^{-1}([\tfrac{c}{2},+\infty))$. 

Fix a point $x\in U_{\ell,k}\cap M$. Then $0<x_1<1/k$ and $\sum_{j=2}^m(x_j-\gamma_{j\ell}(x_1))^2 < x_1^{2\ell+2}$. Thus, $|x_j-\gamma_{j\ell}(x_1)|<x_1^{\ell+1}$ for $j=2,\ldots,m$ and so $x_j=\gamma_{j\ell}(x_1)+\rho_jx_1^{\ell+1}$ for some $\rho_j\in\R$ such that $|\rho_j|<1$. Write $\rho:=(0,\rho_2,\ldots,\rho_m)$ and observe that by \ref{piccolo2}
$$
g_i(x)=g_i(\gamma_{\ell}(x_1))+x_1^{\ell+1}H_i(x_1^{\ell+1},\gamma_{\ell}(x_1),\rho)\\
>g_i(\gamma_{\ell}(x_1))-x_1^{\ell+1}C_\ell>0;
$$
hence, $x\in\{g_1>0,\ldots,g_r>0\}\cap M\subset V\subset\widehat{g}^{-1}([\tfrac{c}{2},+\infty))$.

Now we check $U_{\ell,k}\cap\partial M\subset\widehat{g}^{-1}([\tfrac{c}{2},+\infty))$. As $U_{\ell,k}$ is open in $\betaa M$ and $\widehat{\partial}M$ is dense in $\partial M$ (see Lemma \ref{deltam}), it is enough to show that $U_{\ell,k}\cap\widehat{\partial}M$ is contained in $\widehat{g}^{-1}([\tfrac{c}{2},+\infty))$. To that end it is sufficient to prove that $\mu\in\{g_1>0,\ldots,g_r>0\}_{F_1}$ for all formal paths $\mu\in (U_{\ell,k}\cap M)_{F_1}$. Indeed, $\mu_1(\t)>0$ because $\mu\in M_{F_1}$ and $M\subset\{y_1>0\}$. After reparameterizing we may assume $\mu_1(\t)=\t^p$ for some $p\geq 1$. Since $\mu\in(U_{\ell,k})_{F_1}$, we get $\|\mu(\t)-\gamma_\ell(\t^p)\|^2<\t^{2(\ell+1)p}$, so $\|\mu(\t)-\gamma_{\ell}(\t^p)\|^2\in(\t)^{2(\ell+1)p}$. As $\|\alpha(\t)-\gamma_{\ell}(\t)\|^2\in(\t)^{2(\ell+1)}$, we deduce $\|\mu(\t)-\alpha(\t^p)\|^2\in(\t)^{2(\ell+1)p}$ and therefore by \ref{piccolo} $g_i(\mu(\t))>0$ for each index $i=1,\ldots,r$, that is, $\mu\in\{g_1>0,\ldots,g_r>0\}_{F_1}$.

We conclude $U_{\ell,k}=(U_{\ell,k}\cap M)\cup(U_{\ell,k}\cap\partial M)\subset\widehat{g}^{-1}([\tfrac{c}{2},+\infty))\subset{\mathcal D}_{\betaa M}(g)$, as required.
\end{proof}

\begin{cor}\label{neigh0}
Let $h\in{\mathcal S}^*(M)$, $\widehat{h}:\betaa M\to\R$ be its unique continuous extension to $\betaa M$ and $H:=\widehat{h}|_{\partial M}:\partial M\to\R$. Then
\begin{itemize}
\item[(i)] The set $Z_{\partial M}(H)$ is a closed neighborhood in $\partial M$ of each free maximal ideal $\gtm_\alpha^*\in(\widehat{\partial}M\setminus\cl_{\betaa M}(\rho_1(M)))\cap Z_{\partial M}(H)$. 
\item[(ii)] $Z_{\partial M}(H)=\cl_{\partial M}(\Int_{\partial M}(Z_{\partial M}(H)))$.
\item[(iii)] If $Z_{\partial M}(H)$ is a singleton $\{\gtm^*\}$, then $\gtm^*$ is an endpoint of $\betaa M$ and it belongs to $\cl_{\betaa M}(\ol{M}_{\leq1})\setminus\ol{M}_{\leq1}$.
\end{itemize}
\end{cor}
\begin{proof}
(i) Consider the map $\betaa\,{\tt j}:\betaa M_{\lc}\to\betaa M$ induced by the inclusion ${\tt j}:M_{\lc}\hookrightarrow M$. Recall that if $Y:=\rho_1(M)$, then by \cite[6.7(ii)]{fg3} the restriction 
$$
\betaa\,{\tt j}|:\betaa M_{\lc}\setminus(\betaa\,{\tt j})^{-1}(\cl_{\betaa M}(Y))\to\betaa M\setminus\cl_{\betaa M}(Y)
$$ 
is a homeomorphism. Consequently, it holds
$$
\widehat{\partial}M\setminus\cl_{\betaa M}(Y)=\betaa\,{\tt j}(\widehat{\partial}M_{\lc}\setminus(\betaa\,{\tt j})^{-1}(\cl_{\betaa M}(Y))).
$$

Let $\widehat{h\circ{\tt j}}=\widehat{h}\circ(\betaa\,{\tt j}):\betaa M_{\lc}\to\R$ be the (unique) continuous extension of $h\circ{\tt j}$ to $\betaa M_{\lc}$ and consider its restriction 
$$
\widehat{h\circ{\tt j}}|_{\partial M_{\lc}}=\widehat{h}\circ(\betaa\,{\tt j})|_{\partial M_{\lc}}:\partial M_{\lc}\to\R. 
$$
By Corollary \ref{ppp} and \cite[6.1]{fe1} we deduce that $Z_{\partial M_{\lc}}(\widehat{h\circ{\tt j}}|_{\partial M_{\lc}})$ is a neighborhood of $\gtn_\alpha^*:=(\betaa\,{\tt j})^{-1}(\gtm^*_\alpha)$ in $\partial M_{\lc}$. In addition, $\gtn_\alpha^*\not\in(\betaa\,{\tt j})^{-1}(\cl_{\betaa M}(Y))$ because $\gtm^*_\alpha\not\in\cl_{\betaa M}(Y)$. Therefore $Z_{\partial M_{\lc}}(\widehat{h}\circ(\betaa\,{\tt j})|_{\partial M_{\lc}})\setminus(\betaa\,{\tt j})^{-1}(\cl_{\betaa M}(Y))$ is a neighborhood of $\gtn_\alpha^*$ in $\partial M_{\lc}\setminus(\betaa\,{\tt j})^{-1}(\cl_{\betaa M}(Y))$. Taking images under $\betaa\,{\tt j}$, we conclude: \em $Z_{\partial M}(H)$ is a closed neighborhood of $\gtm_\alpha^*$ in $\partial M$\em.

(ii) We prove the non-obvious inclusion in (ii). Let $\gtm^*\in Z_{\partial M}(H)$ and $g\in{\mathcal S}^*(M)$ be such that $\gtm^*\in \di_{\betaa M}(g)$. We must prove that $\di_{\betaa M}(g)$ meets $\Int_{\partial M}(Z_{\partial M}(H))$. By \cite[4.10]{fg5} there exists $b\in{\mathcal S}^*(M)$ whose continuous extension $\widehat{b}$ to $\betaa M$ satisfies $\widehat{b}(\gtm^*)=1$ and $\widehat{b}|_{\cl_{\betaa M}(Y)}=0$. Let $\widehat{g}:\betaa M\to\R$ be the unique continuous extension of $g$ to $\betaa M$. By Lemma \ref{deltam}(i) there exists $\gtm_\alpha^*\in\widetilde{\partial}M$ such that $\widehat{h}(\gtm_\alpha^*)=\widehat{h}(\gtm^*)=0$, $\widehat{g}(\gtm_\alpha^*)=\widehat{g}(\gtm^*)\neq 0$ and $\widehat{b}(\gtm_\alpha^*)=\widehat{b}(\gtm^*)= 1$. Consequently, $\gtm_\alpha^*\in\widetilde{\partial}M\setminus\cl_{\betaa M}(Y)\subset\widehat{\partial}M\setminus\cl_{\betaa M}(Y)$ and $\gtm_\alpha^*\in\di_{\betaa M}(g)\cap Z_{\partial M}(H)$. Using (i), this implies $\gtm_\alpha^*\in\Int_{\partial M}(Z_{\partial M}(H))$ and we are done. 

(iii) We assume that $M$ is bounded. By Lemma \ref{deltam}(i) there exists a semialgebraic path $\alpha\in\R[[\t]]_{\rm alg}$ such that $\gtm^*=\gtm_{\alpha}^*$. We may assume that $\alpha$ defines a semialgebraic homeomorphism $\alpha:[0,1]\to M\cup\{\alpha(0)\}$. By (ii) we know that $\{\gtm_{\alpha}^*\}$ is an open and closed subset of $\partial M$, so it is a connected component of $\partial M$. Let $U$ be a closed neighborhood of $\gtm_{\alpha}^*$ in $\betaa M$ such that $\partial M\cap U=\{\gtm_{\alpha}^*\}$ and $U\cap M$ is a semialgebraic set. We have $U=(U\cap M)\cup\{\gtm_{\alpha}^*\}$ and $\partial U=\{\gtm_{\alpha}^*\}$. Let $N:=\alpha([0,1))$ and assume $N\subset U$. As $\partial U\subset\cl_{\betaa M}(N)$, we conclude by \cite[5.7]{fg5} that after shrinking $U$, we may assume $U=N\cup\{\gtm_{\alpha}^*\}$. Consequently, $\gtm_{\alpha}^*$ is an endpoint of $\betaa M$ and by Corollary \ref{eum}(i) $\gtm_{\alpha}^*\in\cl_{\betaa M}(\ol{M}_{\leq1})\setminus\ol{M}_{\leq1}$, as required.
\end{proof}
\begin{cor}\label{eum2}
The set $\Ee$ of maximal ideals of ${\mathcal S}^*(M)$ that are the Jacobson radical of a principal ideal of ${\mathcal S}^*(M)$ equals $M\cup\eta(\betaa M)$.
\end{cor}
\begin{proof}
Assume $M$ is bounded. The inclusion $\Ee\subset M\cup\eta(\betaa M)$ follows from Corollary \ref{neigh0}(iii). To prove the converse inclusion observe that $M\subset\Ee$ (the Jacobson radical of the ideal generated by the distance function to a point $a\in M$ restricted to $M$ is $\gtm^*_a$). Pick a point $\gtm^*\in\eta(\betaa M)\setminus M$. By Corollary \ref{neigh0}(iii) $\gtm^*\in\cl_{\betaa M}(\ol{M}_{\leq1})\setminus\ol{M}_{\leq1}$. Let $U\subset\betaa M$ be a compact neighborhood of $\gtm^*$ equipped with a homeomorphism $\xi:U\to[0,1]$ such that $\xi(\gtm^*)=0$, the difference $U\setminus\{\gtm^*\}\subset M$ is a semialgebraic set and $\xi|_{U\setminus\{\gtm^*\}}$ is a semialgebraic map. Let $g:M\to[0,1]$ be the bounded semialgebraic function given by
$$
g(x):=\begin{cases}
1&\text{if $x\in M\setminus U$,}\\
\xi(x)&\text{if $x\in U$.}
\end{cases}
$$
The Jacobson radical of the principal ideal $g{\mathcal S}^*(M)$ is $\gtm^*$, as required. 
\end{proof}

\subsection{Points of the maximum spectrum with metrizable neighborhoods}
We end this section with the announced characterization of the points of the semialgebraic Stone-C\v{e}ch compactification $\beta M$ that have metrizable neighborhoods.

\begin{thm}\label{mn}
The set of points of $\betaa M$ that have a metrizable neighborhood in $\betaa M$ equals $M_{\lc}\cup(\cl_{\betaa M}(\ol{M}_{\leq1})\setminus\ol{M}_{\leq1})$. 
\end{thm}
\begin{proof}
Let $\Tt$ be the set of points of $\betaa M$ that have a metrizable neighborhood in $\betaa M$ and let $q\in M_{\lc}\cup(\cl_{\betaa M}(\ol{M}_{\leq1})\setminus\ol{M}_{\leq1})$. If $q\in\cl_{\betaa M}(\ol{M}_{\leq1})\setminus\ol{M}_{\leq1}$, then $q$ has by Corollary \ref{eum}(i) an open neighborhood in $\betaa M$ that is homeomorphic to $[0,1)$, which is a metrizable space. On the other hand, $M_{\lc}$ is open in $\betaa M$ and a metrizable neighborhood of all its points. Therefore $M_{\lc}\cup(\cl_{\betaa M}(\ol{M}_{\leq1})\setminus\ol{M}_{\leq1})\subset \Tt$. Let us prove $\Tt\subset M_{\lc}\cup(\cl_{\betaa M}(\ol{M}_{\leq1})\setminus\ol{M}_{\leq1})$ next. 

Let $p\in \Tt$ and $W$ be a metrizable neighborhood of $p$ in $\betaa M$. Let $f\in{\mathcal S}^*(M)$ be such that $p\in{\mathcal D}_{\betaa M}(f)\subset W$. We may assume that the unique continuous extension $\widehat{f}:\betaa M\to\R$ of $f$ to $\betaa M$ satisfies $\widehat{f}(p)=c>0$ and we consider the closed semialgebraic subset $Z:=f^{-1}([\tfrac{c}{2},+\infty))=\widehat{f}^{-1}([\tfrac{c}{2},+\infty))\cap M$ of $M$. 

By \ref{closedbeta1}(i) $\betaa Z$ is homeomorphic to $\cl_{\betaa M}(Z)$ and by Lemma \ref{fc} $\cl_{\betaa M}(Z)\cong\betaa Z$ is a neighborhood of $p$ in $\betaa M$. It contains $p$ and is metrizable because it is a subset of $W$. Hence, by \cite[5.17]{fg5} the subset $Z_{\geq2}$ of points of local dimension $\geq2$ of $Z$ is compact. We write $Z=Z_{\geq2}\cup\ol{Z}_{\leq1}$. Observe that $\ol{Z}_{\leq1}$ is a closed subset of $M$. By \ref{closedbeta1}(i) we can identify $\betaa Z=\betaa Z_{\geq2}\cup\betaa\,\ol{Z}_{\leq1}=Z_{\geq2}\cup\betaa\,\ol{Z}_{\leq1}$. We distinguish two cases:

\noindent{\em Case \em 1.} If $p\in Z_{\geq2}$, then $p\in M$ and $f(p)=c$. As $Z_{\geq2}$ is compact, $Z$ is by Remark \ref{curvas}(iii) a locally compact neighborhood of $p$ in $M$. By Proposition \ref{rho} $p\in M_{\lc}$.

\noindent{\em Case \em 2.} If $p\not\in Z_{\geq2}$, then $p\in\betaa\,\ol{Z}_{\leq1}\setminus Z_{\geq2}$ and $\betaa\,\ol{Z}_{\leq1}\setminus Z_{\geq2}$ is a neighborhood of $p$ in $\betaa M$. There are two possibilities:

(a) $p\in\ol{Z}_{\leq1}$, so $\ol{Z}_{\leq1}\setminus Z_{\geq2}=(\betaa\,\ol{Z}_{\leq1}\setminus Z_{\geq2})\cap M$ is a locally compact neighborhood of $p$ in $M$ and $p\in M_{\lc}$.

(b) $p\in\partial\ol{Z}_{\leq1}=\betaa\,\ol{Z}_{\leq1}\setminus\ol{Z}_{\leq1}\subset\betaa M\setminus M$. As $\partial\ol{Z}_{\leq1}$ is by \cite[5.17]{fg5} a finite set, $\{p\}$ is a connected component of $\partial M$, that is, $p$ is a closed point of $\betaa M$ that is isolated for the inverse topology of $\Speca(M)$. By \ref{restricti0} and Corollary \ref{eum2} $p\in M\cup\eta(\betaa M)$. As $p\not\in M$, we conclude by Corollary \ref{eum}(i) that $p\in\cl_{\betaa M}(\ol{M}_{\leq1})\setminus\ol{M}_{\leq1}$, as required.
\end{proof}

\appendix
\section{Non-semialgebraic homeomorphism between semialgebraic sets}

The behavior of a non-semialgebraic homeomorphism between semialgebraic sets can turn out to be unpredictable with respect to its possible extensions to the semialgebraic Stone--\v{C}ech compactification. In fact, semialgebraic paths become useless in the absence of semi\-algebraicity.

\begin{examples}\label{gpn}
(i) Let $M:=\R^{2}\setminus\{(0,0)\}$ and consider the smooth path $\gamma:(0,1]\to M,\ t\mapsto(t,\lambda\exp(-1/t))$ where $\lambda$ is a fixed positive real number. Then 
\begin{itemize}
\item[(1)] For all $f\in{\mathcal S}^*(M)$ there exists $\lim_{t\to0^+}\,(f\circ\gamma)(t)\in\R$.
\item[(2)] The set $\gtm^*:=\{f\in{\mathcal S}^*(M):\, \lim_{t\to0^+}\,(f\circ\gamma)(t)=0\}$ is a maximal ideal of ${\mathcal S}^*(M)$.
\item[(3)] $\gtm^*=\gtm^*_\alpha$ where $\alpha:(0,1]\to M,\ t\mapsto(t,0)$.
\end{itemize}
\begin{proof}
Let $f\in{\mathcal S}^*(M)$ and $\widehat{f}:\betaa M\to\R$ be the unique continuous extension of $f$ to $\betaa M$. Assume $\widehat{f}(\gtm^*_\alpha)=0$ and observe that statements (1), (2) and (3) are straightforward consequences of the following equality that we prove next.
\begin{equation}\label{evaluando}
\lim_{t\to0^+}f(t,\lambda\exp(-1/t))=0=\lim_{t\to0^+}f(t,0)=\widehat{f}(\gtm^*_\alpha).
\end{equation}

By Corollary \ref{neigh0} the set $Z_{\partial M}(\widehat{f})$ is a closed neighborhood of $\gtm_\alpha^*$ in $\partial M=\betaa M\setminus M$. Thus, there exists $g\in{\mathcal S}^*(M)$ such that $\gtm_{\alpha}^*\in{\mathcal D}_{\betaa M}(g)\cap\partial M\subset Z_{\partial M}(\widehat{f})$. We may also assume $c:=\widehat{g}(\gtm_\alpha^*)>0$. Consider the closed semialgebraic set $Z:=g^{-1}([\tfrac{c}{2},+\infty))\cap\{x^2+y^2\leq1\}$. As $\gtm^*_\alpha\in{\mathcal D}_{\betaa M}(g)\cap\partial M$, there exists $\veps>0$ such that $Y_\veps:=(0,\veps]\times\{0\}\subset Z$. 

Otherwise, as $Z$ is semialgebraic, there exists $\veps>0$ such that the closed semialgebraic subsets $Z$ and $Y=Y_\veps$ of $M$ are disjoint. Then there exists by \cite{dk} a semialgebraic function $h\in{\mathcal S}^*(M)$ such that $h|_Z=0$ and $h|_Y=1$. Thus, $\widehat{h}(\gtm_\alpha^*)=1$ and by \ref{closedbeta1}(i)\&(ii) and Lemma \ref{fc} we obtain
$$
\gtm^*_\alpha\not\in {\mathcal Z}_{\betaa M}(h)\supset\cl_{\betaa M}(Z)=\cl_{\betaa M}(\widehat{g}^{-1}([\tfrac{c}{2},+\infty))\cap\cl_{\betaa M}(x^2+y^2\leq 1),
$$
which is a contradiction.

As the Taylor series at the origin of $\lambda\exp(-1/t)$ is identically zero, the image of $\gamma|:(0,\delta]\to M$ for $\delta>0$ small enough is contained in $Z$. As $\widehat{g}^{-1}([\tfrac{c}{2},+\infty))\cap\partial M\subset Z_{\partial M}(\widehat{f})$, the closure of the graph $T$ of $f|_Z$ in $\R^3$ is $T\cup\{(0,0,0)\}$. Consequently, equality \eqref{evaluando} holds, as required.
\end{proof}

(ii) Let $M:=\R^2\setminus\{(0,0)\}$ and consider the homeomorphism $\varphi:M\to M$ given by
$$
(x,y)\mapsto
\left\{\begin{array}{ll}
\!\!(x,(1-\frac{\exp(-1/x)}{x})(2y-x)+\exp(-1/x))&\text{ if $0\leq\frac{1}{2}x\leq y\leq x$,}\\[4pt]
\!\!(x,\frac{2\exp(-1/x)}{x}y)&\text{ if $0\leq y \leq\frac{1}{2}x$,}\\[5pt]
\!\!(x,y)&\text{ if $y\leq 0$ or $x\leq y$.}
\end{array}
\right.
$$
Note that $\varphi(t,\mu t)=(t,2\mu\exp(-1/t))$ for $0\leq\mu\leq 1/2$ and $t>0$. Thus, the homeomorphism $\varphi:M\to M$ cannot be extended to a homeomorphism $\widehat{\varphi}:\betaa M\to\betaa M$ because by (i) such an extension would map the (distinct) maximal ideals $\gtm_{\mu}^*:=\{f\in{\mathcal S}^*(M):\, \lim_{t\to0^+}f(t,\mu t)=0\}$, where $0\leq\mu\leq 1/2$, onto the maximal ideal $\gtm_{\alpha}^*$ described in (i.3).
\end{examples}

\bibliographystyle{amsalpha}

\end{document}